\documentclass[a4paper,twoside]{article}      
\usepackage{amsmath,amssymb,amsfonts,amsthm,amscd} 
\usepackage{graphics}                 
\usepackage{color}                    
\usepackage{hyperref}                
\usepackage{ mathrsfs }
\usepackage[margin=2cm]{geometry}
\usepackage{indentfirst}
\usepackage{bbold}
\usepackage{enumerate}
\usepackage{url}         
\usepackage{colonequals} 
\usepackage{authblk} 
\usepackage{a4wide}

\usepackage{upgreek}

\newcommand{\Kl}{\left(}
\newcommand{\Kr}{\right)}
\newcommand{\El}{\left[}
\newcommand{\Er}{\right]}

\newcommand{\Ra}{\rightarrow}

\newcommand{\hh}{\hspace{4pt}}

\newcommand{\RR}{\mathbf{R}}

\newcommand{\N}{\mathcal{N}}

\def\R{\mathbf{ R}}

\def\P{\mathcal{P}}

\def\E{\mathcal{E}}

\def\L{\mathcal{L}}

\newcommand{\Keywords}[1]{\par\indent
{\small{\textbf{Key words.} \/} #1}}
\newcommand{\Ptwo}{\mathcal{P}_2}

\newcommand{\abs}[1]{\left\lvert#1\right\rvert}

\renewcommand{\det}{\mathop{\mathrm{det}}}
\newcommand{\tr}{\mathop{\mathrm{tr}}}

\newtheorem{theorem}{Theorem}[section]

\newtheorem{proposition}[theorem]{Proposition}

\newtheorem{remark}[theorem]{Remark}
\newenvironment{definition}[1][Definition]{\begin{trivlist}
\item[\hskip \labelsep {\bfseries #1}]}{\end{trivlist}}

\begin{document}

\title{Asymptotic equivalence of the discrete variational functional and a rate-large-deviation-like functional in the Wasserstein gradient flow of the porous medium equation}
\author{Manh Hong Duong}

\affil{Department of Mathematics and Computer Sciences, Technische Universiteit Eindhoven, The Netherlands. \\
Email: m.h.duong@tue.nl}
\date{\today}
%
%

\maketitle

\begin{abstract}
In this paper, we study the Wasserstein gradient flow structure of the porous medium equation. We prove that, for the case of $q$-Gaussians on the real line, the functional derived by the JKO-discretization scheme is asymptotically equivalent to a rate-large-deviation-like functional. The result explains why the Wasserstein metric as well as the combination of it with the Tsallis-entropy play an important role.
\end{abstract}


\section{Introduction}
\subsection{The porous medium equation}
In a seminal paper, Otto~\cite{Otto01} shows that the porous medium equation
\begin{equation}
\label{eq:PMEequation}
\partial_t\rho(t,x) =\Delta \rho^{2-q}(t,x),\quad\text{for}\quad(t,x)\in (0,\infty)\times\RR^d\quad \text{and}~\rho(0,x)=\rho_0(x),
\end{equation}
can be interpreted as a gradient flow of the $q$-Tsallis entropy functional\footnote{$E_1$ is the Boltzmann-Shannon entropy},
\begin{equation}
E_q(\rho)=
\begin{cases}\frac{1}{1-q}\int_{\R^d} \rho(x)\El \rho(x)^{1-q}-1\Er dx\quad \text{if}\quad q\neq 1,\\
\int_{\R^d} \rho(x)\log \rho(x)dx\quad \text{if}\quad q=1,
\end{cases}
\end{equation}
 with respect to the Wasserstein distance $W_2$ on the space of probability measures with finite second moment $\P_2(\R^d)$,
\begin{equation}
\label{eq: W_2}
W_2(\mu,\nu)=\Big\{\inf_{P\in\Gamma(\mu,\nu)}\int_{\R^{d}\times\R^d}\! \abs{x-y}^2 P(dxdy)\Big\}^\frac{1}{2},
\end{equation}
where $\Gamma(\mu,\nu)$ is the set of all probability measures on $\R^{d}\times\R^d$ with  marginals $\mu$ and~$\nu$.


This statement can be understood in a variety of different ways. In~\cite{Otto01}, Otto shows this from a differential geometry point of view by considering $(\P_2(\R^d), W_2)$ as an infinite dimensional Riemannian manifold. However, for the purpose of this paper, the most useful way is that the solution $t\mapsto\rho(t,x)$ can be approximated by the JKO-discretization scheme. Let $T>0$ be given and $h>0$ be a time step. The discrete approximated solution $\rho^{(n)}=\rho(nh); n=0,\dots,\lfloor\frac{T}{h}\rfloor$ is defined recursively by (see also~\cite{JKO98,WW10} and~\cite{AGS08} for an exposition of this subject)
\begin{align}
&\rho^{(0)}=\rho_0,\nonumber
\\&\rho^{(n)}\in \mathrm{argmin}_{\rho\in \Ptwo(\RR^d)}K_h(\rho,\rho^{(n-1)}),\quad K_h(\rho,\rho^{(n-1)})=\frac{1}{4h}W_2^2(\rho,\rho^{(n-1)})+\frac{1}{2}( E_q(\rho)-E_q(\rho^{(n-1)})).\label{eq:discrete scheme}
\end{align}
The theory of Wasserstein gradient flow has been developed tremendously since~\cite{JKO98,Otto01}. At least two different approaches have been investigated. The first one is to explore a larger class of evolution equations that are Wasserstein (gernalized/modified) gradient flows. Many equations are now proven to belong to this class~\cite{GO01,Gla03,Agu05,GST09,MMS09,Lis09,FG10,Mie11b,LMS12}. Recently some attempts have been made to extend the theory to discrete setting~\cite{Maa11,Mie12} or to systems that contain also conservative behavior\cite{Hua00,HJ00,DPZ12TMP}. The second approach is to understand why the Wasserstein metric and the combination of it with the entropy functional appear in the setting. This direction recently has been received a lot of attraction. In~\cite{ADPZ11,DLZ12,DLR12,PRV11,Leo12}, the authors address this question, for the case of the linear diffusion equation, i.e., for $q=1$, by establishing an intriguing connection between a microscopic many-particle model and the macroscopic gradient flow structure of the diffusion equation. They show that the functional $K_h$ in~\eqref{eq:discrete scheme} is asymptotically equivalent as $h\rightarrow 0$ to a discrete rate functional $\widetilde{J}_h$ that comes from the large deviation principle of the microscopic model. We now briefly recall the results in these papers in more details.
The rate functional $\widetilde{J}_h\colon\P_2(\R^d)\rightarrow [0,+\infty]$ is defined by
\begin{equation}
\label{eq:J_h diffusion}
\widetilde{J}_h(\rho|\rho_0)=\inf_{Q\in \Gamma(\rho_0,\rho)}H(Q\|\widetilde{Q}_{0\rightarrow h}),
\end{equation}
where
\begin{equation*}
\widetilde{Q}_{0\rightarrow h}(dxdy)=p_h(x,y)\rho_0(dx)dy;\qquad p_h(x,y)=\frac{1}{(4\pi h)^\frac{d}{2}}e^{-\frac{\abs{x-y}^2}{4h}},
\end{equation*}
and
$H(Q\|\widetilde{Q}_{0\rightarrow h})$ is the relative entropy of $Q$ with respect to $\widetilde{Q}_{0\rightarrow h}$
\begin{equation*}
H(Q\|\widetilde{Q}_{0\rightarrow h})=\int_{\R^{2d}}
\log\Kl\frac{dQ}{d_{Q_{0\rightarrow h}}}\Kr dQ,\quad\text{if}~\frac{dQ}{d_{Q_{0\rightarrow h}}}~\text{exists and otherwise being }~+\infty.
\end{equation*}
Note that in the case of $d=1$ and $\rho_0(x)=\N(\mu_0,\sigma_0^2)$, then $\widetilde{Q}_{0\rightarrow h}$ is a bivariate Gaussian $\N(\mu_0,\sigma_0,\mu_0,\sigma_h^2,\theta_h)$, where $\sigma_h^2=\sigma_0^2+2h,\quad \text{and}\quad \theta_h=\frac{\sigma_0}{\sigma_h}$.

It is proven in~\cite{Leo12}\footnote{This result was appeared first in 2007 in his preprint paper~\cite{Leo07}} that
\begin{equation}
\label{eq:1st order Gamma diffusion eqn}
4h\widetilde{J}_h(\cdot,\rho_0)\overset{\Gamma}{\rightarrow}W_2^2(\cdot,\rho_0)\quad\text{as}\quad h\rightarrow 0,
\end{equation}
where $\Gamma$ is denoted for Gamma-convergence, which is introduced in Section~\ref{sec: review}. This result is improved to the next order in~\cite{ADPZ11,DLZ12,DLR12,PRV11} to give
\begin{equation}
\label{eq: 2nd second order Gamma diffusion eqn}
\widetilde{J}_h(\cdot|\rho_0)-\frac{1}{4h}W_2^2(\cdot,\rho_0)\overset{\Gamma}{\rightarrow}\frac{1}{2}(E_1(\cdot)-E_1(\rho_0))\quad\text{as}\quad h\rightarrow 0.
\end{equation}
\eqref{eq:1st order Gamma diffusion eqn} and~\eqref{eq: 2nd second order Gamma diffusion eqn} indicate the asymptotic Gamma development of the functional $\widetilde{J}_h(\cdot|\rho_0)$
\begin{equation}
\label{eq:Gammadevelopment diff eqn}
\widetilde{J}_h(\rho|\rho_0)\approx\frac{1}{4h}W_2^2(\rho,\rho_0)+\frac{1}{2}(E_1(\rho)-E_1(\rho_0))+o(1)\quad\text{as}\quad h\rightarrow 0.
\end{equation}

It is worth pointing out that the right hand side of~\eqref{eq:Gammadevelopment diff eqn} is exactly the functional $K_h$ in the variational scheme~\eqref{eq:discrete scheme}. It is this asymptotic Gamma development that not only provides the link between the microscopic model and the Wasserstein gradient flow formulation of the linear diffusion equation but also explains why the Wassertein metric as well as the combination of it with the entropy play a role in the setting.

The aim of this paper is to generalise~\eqref{eq:1st order Gamma diffusion eqn}-~\eqref{eq: 2nd second order Gamma diffusion eqn} to the nonlinear porous medium equation for the class of $q$-Gaussian measures in 1D. As we see in section~\ref{sec: review}, this class plays an important role because it is invariant under the semigroup of the porous-medium equation and is isometric to the space of Gaussian measures with respect to the Wasserstein metric. We now describe our result in the next Section.

\subsection{Main result}
Before introducing the main result of the present paper, we recall corresponding concepts for the $q$-Gaussian measures.

The $q$-exponential function and its inverse, the  $q$-logarithmic function, are defined respectively by
\begin{equation}
\label{eq:qexp}
\exp_q(t)=\El 1+(1-q)t\Er_+^\frac{1}{1-q},
\end{equation}
where $[x]_+=\max\{0,x\}$ and by convention $0^a\colonequals\infty$; and
\begin{equation}
\label{eq:qlog}
\log_q(t)=\frac{t^{1-q}-1}{1-q} \qquad \text{for}~ t>0.
\end{equation}
Given $m$, that is specified later on, the $m$-relative entropy between $Q(dx)=f(x)dx $ and $P(dx)=g(x)dx $
\begin{align}
\label{eq:q-relative entropy}
H_m(Q\big\|P)&=\frac{1}{2-m}\int \El f\log_mf-g\log_m g-(2-m)\log_m g(f-g)\Er dx
\\&=\frac{1}{2-m}\int \El f\log_mf+(1-m)g\log_m g-(2-m)f\log_mg\Er dx.
\end{align}
For $v\in \RR^d$ and $V\in\mathrm{Sym}^+(d,\RR)$, which is the set of all symmetric positive definite matrices of size $d$, the $q$-Gaussian measure with mean $v$ and covariance matrix $V$ is
\begin{equation}
\label{eq:qGaussian}
\N_q(v,V)=C_0(q,d)(\det V)^\frac{1}{2}\exp_q\El-\frac{1}{2}C_1(q,d)\langle x-v,V^{-1}(x-v)\rangle\Er\L^d,
\end{equation}
where $\L^d$ is the Lesbesgue measure on $\R^d$ and $C_0(q,d), C_1(q,d)$ are explicit positive constants depending only on $d$ and $q$ and are given in section~\ref{sec: review}.

In particular, the $q$-Gaussian measure in $1$D has density
\begin{equation}
\N_q(\mu,\sigma^2)=\frac{C_0(q,1)}{\sigma}\exp_q\Kl-\frac{1}{2}C_1(q,1)\frac{(x-\mu)^2}{\sigma^2}\Kr.
\end{equation}
While in $2$D, we call $q$-bivariate Gaussian $\N_q(\mu_1,\sigma_1^2,\mu_2,\sigma_2^2,\theta)$ with density
\begin{equation}
\upnu(x,y)=\frac{C_0(q,2)}{\sigma_1\sigma_2\sqrt{1-\theta^2}}\exp_q\left\{-\frac{1}{2}C_1(q,2)\frac{1}{1-\theta^2}\El\frac{(x-\mu_1)^2}{\sigma_1^2}+\frac{(y-\mu_2)^2}{\sigma_2^2}-\frac{2\theta(x-\mu_1)(y-\mu_2)}{\sigma_1\sigma_2}\Er\right\}
\end{equation}
which corresponds to the mean vector $\boldsymbol{\mu}$ and covariance matrix $\Sigma$
\begin{equation*}
\boldsymbol{\mu}=
    \begin{pmatrix}
      \mu_{1} \\
      \mu_{2}
    \end{pmatrix},\quad
  \Sigma=
  \begin{pmatrix}
    \sigma_{1}^{2}& \theta\sigma_{1}\sigma_{2}  \\
    \theta\sigma_{1}\sigma_{2} & \sigma_{2}^{2}
  \end{pmatrix}.
 \end{equation*}
It is shown in section~\ref{sec: review} that, if the initial data is $\rho_0(x)=\N_q(\mu_0,C\sigma_0^2)$, then the solution to the porous medium equation at time $t$ is $\rho(t,x)=\N_q(\mu_0,C(t+\sigma_0^{3-q})^\frac{2}{3-q})$, where $C$ is an explicit constant given in section~\eqref{eq:C-formula}.

We are now in the position to introduce the main result.
\begin{theorem}
Let $d=1,\quad q\in Q_d\equiv (0,1)\cup\Kl1,\frac{d+4}{d+2}\Kr $ and $\N_0=\N_q(\mu_0,C\sigma_0^2)$ be given.

We set $m=3-\frac{2}{q},\quad \sigma_h^2=(h+\sigma_0^{3-q})^\frac{2}{3-q}$ and
\begin{equation}
\label{eq: Q_0h}
Q_{0\Ra h}=\N_{m}(\mu_0,C\sigma_0^2,\mu_0,C\sigma_h^2,\frac{\sigma_0}{\sigma_h}),
\end{equation}
For $\N_q=\N_q(\mu,C\sigma^2)$, we define the functional $J_h(\N_q,\N_0)$ by
\begin{equation}
\label{eq: J_h}
J_h(\N_q|\N_0)\colonequals \inf_{Q\in \mathcal{Q}} H_{m}(Q\big\|Q_{0\Ra h}),
\end{equation}
where
\begin{equation*}
\mathcal{Q}\colonequals \left\{\N_{m}(\mu_0,C\sigma_0^{2},\mu,C\sigma^{2},\theta)\big|~\theta\in\RR\right\}.
\end{equation*}

Then there exist explicit constants $a=a(\sigma_0,q)$ and  $b=b(\sigma_0,q)$, which are given respectively in~\eqref{eq:a} and~\eqref{eq:b}, such that the following statements hold
\begin{enumerate}
\item $a(\sigma_h^2-\sigma_0^2)^\frac{1}{q} J_h(\cdot|\N_0)\overset{\Gamma}{\rightarrow}W_2^2(\cdot,\N_0) \quad\text{as}~ h\rightarrow 0$.
\item $ab(\sigma_h^2-\sigma_0^2)^\frac{1-q}{q}J_h(\cdot|\N_0)-\frac{b}{(\sigma_h^2-\sigma_0^2)}W_2^2(\cdot,\N_0)\overset{\Gamma}{\rightarrow}E_q(\cdot)-E_q(\N_0)\quad\text{as}~ h\rightarrow 0$.
\item If $0<q<1$, then $ab(\sigma_h^2-\sigma_0^2)^\frac{1-q}{q} J_h(\cdot|\N_0)-\frac{1}{2h}W_2^2(\cdot,\N_0)\xrightarrow{\Gamma-\liminf} E_q(\cdot)-E_q(\N_0)\quad\text{as}~ h\rightarrow 0$
\end{enumerate}
\end{theorem}
When $q\rightarrow 1$, then $a\rightarrow 4,~b\rightarrow \frac{1}{2},~\sigma_h^2-\sigma_0^2\rightarrow h$ and we recover~\eqref{eq:1st order Gamma diffusion eqn}-\eqref{eq: 2nd second order Gamma diffusion eqn} for the diffusion equation.

\subsection{Discussion}

Given $\N_q$ and $\N_0$,  both $H_m(\N_q\big\|\N_0)$ and $J_h(\N_q|\N_0)$ are always non-negative and are equal to $0$ if and only if $\N_q=\N_0$. By its definition, $J_h(\N_q|\N_0)$ measures the deviation of $\N_q$ from the solution of the porous medium equation at time $h$ given the initial data $\N_0$. Hence minimizing $J_h$ means to find the best approximation of the solution at a time $h$. As a consequence of the Gamma convergence, minimizers of the functionals converge to minimizer of the Gamma-limit functional. Thus the main theorem explains why the Wasserstein metric is involved and why we should minimize the combination of it with the Tsallis entropy.


For the linear diffusion equation, by Sanov's theorem the relative entropy is the \emph{static} rate functional of the empirical process of many i.i.d particles. While the functional $\widetilde{J}_h$ in~\eqref{eq:J_h diffusion} is the rate functional after time $h$ of the empirical process of many-Brownian motions. Hence it has a clear microscopic interpretation. The $m$-relative entropy $H_m$ and the functional $J_h$ are defined and have similar properties as $H$ and $\widetilde{J}_h$. However, it is unclear whether $J_h$ can be proven to be the rate functional of some microscopic stochastic process. In a recent paper~\cite{BLMV13}, the authors introduce a stochastic particle system, the Ginzburg-Landau dynamics, and show that the porous medium equation is the hydrodynamic limit of the system. Moreover, the $m$-relative entropy is the static rate functional of the invariant measures. It would be interesting to study whether the functional $\widetilde{J}_h$ is actually (or related to) the rate functional at time $h$ of the Ginzburg-Landau dynamics. Another question for future research is to extend the main theorem to a larger class of measures.

\subsection{Organisation of the paper}
The rest of the paper is organized as follows. In section~\ref{sec: review}, we first recall the definition of Gamma convergence and relevant properties of the $q$-Gaussians. Next we compute the functional $J_h$ in section~\ref{sec: minimising relative entropy}. Finally, the proof of the main theorem is given in section~\ref{sec:proof}.
%


\section{Preliminaries}
\label{sec: review}
We first recall the definition of Gamma convergence for the reader's convenience.
\begin {definition}\cite{Bra02}
Let $X$ be a metric space. We say that a sequence $f_n\colon X\rightarrow\overline{\R}~\,\Gamma-$ converges in $X$ to $f\colon X\rightarrow\overline{\R}$, denoted by $f_n\xrightarrow{\Gamma}f$, if for all $x\in X$ we have
\begin{enumerate}
\item (lower bound part) For every sequence $x_n$ converging to $x$
\begin{equation}
\liminf_{n\rightarrow\infty}f_n(x_n)\geq f(x),
\end{equation}
\item(upper bound part) There exists a sequence $x_n$ converging to $x$ such that
\begin{equation}
\lim_{n\rightarrow\infty}f_n(x_n)=f(x).
\end{equation}
\end{enumerate}
\end{definition}
If $f_n$ satisfy the lower (or upper, respectively) bound part then we write $f_n\xrightarrow{\Gamma-\liminf}f$ (or $f_n\xrightarrow{\Gamma-\limsup}f$ respectively).
\subsection{Properties of $q$-Gaussian measures}
The $q$-exponential function and $q$-logarithmic function satisfy the following properties
\begin{equation}
\label{eq:logq-expq}
\exp_q(\log_qx)=\log_q(\exp_qx)=x,
\end{equation}
and
\begin{align}
\label{eq:logq-product}
\log_q(xy)&=\log_qx+\log_qy+(1-q)\log_q(x)\log_qy\nonumber
\\&=\log_qx+x^{1-q}\log_qy
\\&=\log_qy+y^{1-q}\log_qx\nonumber.
\end{align}
The constants $C_0(q,d)$ and $C_1(q,d)$ in~\eqref{eq:qGaussian} are given by
\begin{equation}
\label{eq:C1(q,d)}
C_1(q,d)=\frac{2}{2+(d+2)(1-q)},
\end{equation}
and
\begin{equation}
\label{eq:C0(q,d)}
C_0(q,d)=
\begin{cases} \frac{\Gamma\Kl\frac{2-q}{1-q}+\frac{d}{2}\Kr}{\Gamma\Kl\frac{2-q}{1-q}\Kr} \Kl\frac{(1-q)C_1(q,d)}{2\pi}\Kr^\frac{d}{2}& \text{if $0<q<1$,}
\\ \frac{\Gamma\Kl\frac{1}{q-1}\Kr}{\Gamma\Kl\frac{1}{q-1}-\frac{d}{2}\Kr}\Kl\frac{(q-1)C_1(q,d)}{2\pi}\Kr^\frac{d}{2}& \text{if $1<q<\frac{d+4}{d+2}$.}
\end{cases}
\end{equation}
%
%
%

\subsection{q-Gaussian measures and solutions of the porous medium equation}
The porous medium equation~\eqref{eq:PMEequation} has a self-similar solution, which is called the Barenblatt-Pattle solution, of the form
\begin{align}
\label{eq:stationarysolu}
\rho_q(x,t)&\colonequals \El At^{-d\alpha(1-q)}-B|x|^2t^{-1}\Er_+^\frac{1}{1-q}\nonumber
\\&=\El A-B|x|^2t^{-2\alpha}\Er_+^\frac{1}{1-q}t^{-d\alpha},
\end{align}
where
\begin{equation}
\label{eq:alpha}
\alpha=\alpha(q,d)\colonequals\frac{1}{d(1-q)+2}, B=B(q,d)\colonequals\frac{(1-q)\alpha}{2(2-q)},
\end{equation}
and $A$ is a normalization constant
\begin{equation*}
\int_\RR^d\rho_q(x,t)\L^d(dx)=1.
\end{equation*}
More precise,
\begin{equation}
\label{eq:A}
A\colonequals C_0^{2\alpha(1-q)}\El\frac{\alpha}{(2-q)C_1}\Er^{d\alpha(1-q)}.
\end{equation}
It is straightforward to see that
\begin{equation*}
\rho_q(x,t)=\N_q(0,Ct^{2\alpha}I_d),
\end{equation*}
where
\begin{equation}
\label{eq:C-formula}
C\colonequals\frac{(2-q)C_1}{\alpha}A.
\end{equation}
It is well known that a solution to the diffusion equation is obtained by a convolution of an initial data with the diffusion kernel. Hence if the initial data is a Gaussian measure $\N(v,V)$ then the solution at time $t$ is again a Gaussian $\N(v,V_t)$, which is given by
\begin{equation*}
\N(v,V_t)=\N(v,V)*\N(0,2tI_d)=\N(v,V+2tI_d).
\end{equation*}
In \cite{Tak12,OW10}, the authors show that a similar statement holds for the porous medium equation on the space of $q$-Gaussian measures. That is, a solution to the porous medium equation with an initial data being a $q$-Gaussian measure is again a $q$-Gaussian for all time. Moreover, to find a solution at time $t>0$, it reduces to solving an ordinary differential equation for the covariance matrix. Let $\Theta$ be a map on $\mathrm{Sym}^+(d,\RR)$ defined by
\begin{equation}
\label{eq:mapTheta}
\Theta(V)\colonequals (\det V)^{-\alpha(1-q)}V.
\end{equation}
We note that $\rho_q(x,t)=\N_q(0,C\Theta(tI_d))$.
\begin{proposition}[\cite{Tak12}]
\label{theo:ODE-covriancematrix}
For any $q\in Q_d$ and $V\in \mathrm{Sym}^+(d,\RR)$, we set the time-dependent matrix $V_t$ as
\begin{equation}
\label{eq:ODEforVt}
\Theta(V_t)=\Theta(V)+\sigma(t)I_d, \qquad \frac{d}{dt}\sigma(t)=2\alpha(\det\Theta(V_t))^{-\frac{1-q}{2}}.
\end{equation}
Then $\N_q(v,C\Theta(V_t))$ is a solution to the porous medium equation.
\end{proposition}
\begin{remark}
The assertion also holds true for $q=1$.
\end{remark}
We will work out this theorem in $1$D in more detail. We calculate relevant variables.
\begin{equation}
\alpha =\frac{1}{3-q};\quad \Theta(V)=(\det V)^{-\alpha(1-q)}V=V^{1-\alpha(1-q)}=V^\frac{2}{3-q};\quad \det\Theta(V_t)=V_t.
\end{equation}
The ODE becomes
\begin{equation*}
\Theta(V_t)=\Theta(V)+\sigma(t)I_d, \qquad \frac{d}{dt}\sigma(t)=2\alpha(\Theta(V_t))^{-\frac{1-q}{2}}.
\end{equation*}
Solving this ODE we get
\begin{equation*}
\Theta(V_t)=\El\alpha(3-q)t+\Theta(V)^\frac{3-q}{2}\Er^\frac{2}{3-q}=(t+V)^\frac{2}{3-q}.
\end{equation*}
So if $\rho_0(x)=\N_q(v,CV^\frac{2}{3-q})$ then $\rho(t,x)=\N_q(v,C(t+V)^\frac{2}{3-q})$ for all $t>0$.
In other words, if $\rho_0(x)=\N_q(\mu_0,C\sigma_0^2)$ then $\rho(t,x)=\N_q(\mu_0,C(t+\sigma_0^{3-q})^\frac{2}{3-q})$ for all $t>0$.
\subsection{q-Gaussian and the Wasserstein metric}
The $q$-Gaussian measures have another important property stated in the following Proposition.
\begin{proposition}[\cite{Tak12}]
For any $q\in Q_d$, the space of $q$-Gaussian measures is convex and isometric to the space of Gaussian measures with respect to the Wasserstein metric.
\label{theo:qGaussian-Wasserstein}
\end{proposition}
Hence
\begin{equation*}
W_2(\N_q(\mu,\Sigma),\N_q(\nu,V))^2=W_2(\N(\mu,\Sigma),\N(\nu,V))^2=|\mu-\nu|^2+\mathrm{tr}\Sigma+\mathrm{tr}V-2\mathbf{tr}\sqrt{\Sigma^\frac{1}{2}V^\frac{1}{2}}.
\end{equation*}
In particular
\begin{equation}
\label{eq:Wassertein1D}
W_2(\N_q(\mu_1,\sigma_1^2),\N_q(\mu_2,\sigma_2^2))^2=(\mu_1-\mu_2)^2+(\sigma_1-\sigma_2)^2.
\end{equation}

\section{Computing the functional $J_h$}
\label{sec: minimising relative entropy}
In this section, we compute the functional $J_h$ explicitly. The following Proposition is true for any dimensional $d$. For clarity, we use a boldface font for vectors.
\begin{proposition}
\label{theo:qstar-min}
Let $P(d\mathbf{x})=g(\mathbf{x})d\mathbf{x}$ be absolutely continuous  with respect to the Lebesgue measure in $\RR^{d}$. Let $\mathcal{Q}$ be the set of all Borel measures $Q=f(\mathbf{x})d\mathbf{x}$ in $\RR^{d}$ satisfying
\begin{equation}
\label{eq:res}
\int r_{i}(\mathbf{x})\cdot f(\mathbf{x}) d \mathbf{x}
= a_{i} \hh ,\hh i\in \{1,2,...N\},
\end{equation}
where $r_{i}(\mathbf{x})\colon\RR^{d}\rightarrow\RR $ are given functions and   $a_{i}\in\RR$.

Assume that there is a measure $Q^{*}\in\mathcal{Q}$ that has a density satisfying the equation
\begin{equation}
\label{eq:optimal density}
\log_m f^{*}(\mathbf{x})
=\log_m g(\mathbf{x})+ \sum_{i=1}^{N}\lambda_{i}r_{i}(\mathbf{x}),
\end{equation}
for some $\lambda_{i}\in\RR$.  Then there holds
\begin{equation}
\label{eq:additivity of relative entropy}
H_m(Q\big\|P)=H_m(Q^*\big\|P)+H_m(Q\big\|Q^*)\quad \text{for all}~ Q\in \mathcal{Q}.
\end{equation}
and as a consequence, $Q^{*}$ is the unique minimiser  of $H_m(Q\big\|P)$ over all $Q\in\mathcal{Q}$.
\end{proposition}
\begin{proof}
We have
\begin{align}
(2-m)H_m(Q\big\|P)&\overset{\eqref{eq:q-relative entropy}}{=}\int f\log_mf+(1-m)g\log_mg-(2-m)f\log_mg\nonumber
\\&=\int f\log_mf+(1-m)f^*\log_mf^*-(2-m)f\log_mf^*\nonumber
\\&\qquad+\int(2-m)f(\log_mf^*-\log_mg)-(1-m)(f^*\log_mf^*-g\log_mg)\nonumber
\\&\overset{\eqref{eq:q-relative entropy}}{=}(2-m)H_m(Q^*\big\|Q)+\int(2-m)f(\log_mf^*-\log_mg)-(1-m)(f^*\log_mf^*-g\log_mg)\label{eq:1}
\end{align}
We rewrite the second line in~ \eqref{eq:1} using~ \eqref{eq:res} and~ \eqref{eq:optimal density}
\begin{align}
&\int (2-m)f(\log_mf^*-\log_mg)-(1-m)(f^*\log_mf^*-g\log_mg)\nonumber
\\&\qquad\overset{\eqref{eq:res}}{=}(2-m)\int f\sum_{i=1}^N\lambda_ir_i-(1-m)\int(f^*\log_mf^*-g\log_mg)\nonumber
\\&\qquad=(2-m)\sum_{i=1}^N\lambda_i\int fr_i-(1-m)\int(f^*\log_mf^*-g\log_mg)\nonumber
\\&\qquad\overset{\eqref{eq:optimal density}}{=}(2-m)\sum_{i=1}^N\lambda_i\int f^*r_i-(1-m)\int(f^*\log_mf^*-g\log_mg)\nonumber
\\&\qquad=(2-m)\int f^*\sum_{i=1}^N\lambda_ir_i-(1-m)\int(f^*\log_mf^*-g\log_mg)\nonumber
\\&\qquad\overset{\eqref{eq:optimal density}}{=}(2-m)\int f^*(\log_mf^*-\log_mg)-(1-m)\int(f^*\log_mf^*-g\log_mg)\nonumber
\\&\qquad=\int f^*\log_mf^*-(2-m)f^*\log_mg+(1-m)g\log_mg\nonumber
\\&\qquad \overset{\eqref{eq:q-relative entropy}}{=}(2-m)H_m(Q^*\big\|P)\label{eq:2}
\end{align}
From~ \eqref{eq:1} and~ \eqref{eq:2} we obain \eqref{eq:additivity of relative entropy}.
\end{proof}
\begin{remark}
The property that the relative entropy and the $m$-relative entropy satisfy a generalized Pythagorean relation is well-known in the literature, see for instance~\cite{Csi75} and~\cite{OW10} for similar relationship.
\end{remark}
We now apply this Proposition to the $q$-bivariate measures.
\begin{proposition}
\label{theo:minimizer} Let $P=\N_{m}(\mu_1,\sigma_1^2,\mu_2,\sigma_2^2,\theta)$ be given. We define
\begin{equation}
\label{eq:q2 Gaussian set}
\mathcal{Q}\colonequals \left\{\N_{m}(\nu_1,\xi_1^2,\nu_2,\xi_2^2,\widetilde{\theta})\big|~ \widetilde{\theta}\in\RR\right\}
\end{equation}
The minimizer of the minimization problem
\begin{equation}
\min_{Q\in \mathcal{Q}} H_{m}(Q\big\|P),
\end{equation}
is given by
\begin{equation}
\label{eq: Q*}
Q^*=\N_{m}(\nu_1,\xi_1^2,\nu_2,\xi_2^2,\eta),
\end{equation}
where
\begin{equation}
\label{eq:theta *}
\frac{\eta}{(1-\eta^2)^\frac{3-m}{2}}=
\frac{\theta}{(1-\theta^2)^\frac{3-m}{2}}
\Kl\frac{\xi_1\xi_2}{\sigma_{1}\sigma_{2}}\Kr^{2-m}.
\end{equation}
\end{proposition}
\begin{proof}
We actually prove a stronger statement, namely that $Q^{*}$ is the   minimiser of $H(\cdot\|P)$ over $\mathcal{\widetilde{Q}}$, which is defined as  the set of all $Q(d\mathbf{x})=f(\mathbf{x})d\mathbf{x}$ satisfying
  \begin{align*}
    \int{f(x,y)dxdy}&=1,\hh
    & &
    \\ \int{xf(x,y)dxdy}&=\nu_1,\hh                              &\int{yf(x,y)dxdy}&=\nu_2,
\\ \int{x^{2}f(x,y)dxdy}&=\xi_1^2+\nu_1^2,\hh   &\int{y^{2}f(x,y)dxdy}&=\xi_2^2+\nu_2^2.
\end{align*}
Since $Q^*\in\mathcal{Q}\subset\mathcal{\widetilde{Q}}$, $Q^*$ is also the unique minimizer on $\mathcal{Q}$.
Let $g(\mathbf{x})$ and $f^*(\mathbf{x})$ respectively denote the densities of $P$ and $Q^*$. By Proposition~\ref{theo:qstar-min}, $f^*$ satisfies the following equation
\begin{equation}
\label{eq:to get f*}
\log_m f^{*}(\mathbf{x})
=\log_m g(\mathbf{x})+ \sum_{i=1}^{5}\lambda_{i}r_{i}(\mathbf{x}),
\end{equation}
where $r_1=1, r_2=x, r_3=y, r_4=x^2, r_5=y^2$. To get the equation \eqref{eq:theta *}, we simply equalise the coefficients of $xy$ in two sides of~ \eqref{eq:to get f*} and obtain
\begin{equation*}
 \frac{\theta}{\sigma_1\sigma_2(1-\theta^2)}\El\frac {C_0(m,d)}{\sigma_1\sigma_2\sqrt{1-\theta^2}}\Er^{1-m}= \frac{\eta}{\xi_1\xi_2(1-\eta^2)}\El\frac {C_0(m,d)}{\sigma_1\sigma_2\sqrt{1-\eta^2}}\Er^{1-m},
\end{equation*}
which is equivalent to
\begin{equation*}
\frac{\eta}{(1-\eta^2)^\frac{3-m}{2}}=
\frac{\theta}{(1-\theta^2)^\frac{3-m}{2}}
\Kl\frac{\xi_1\xi_2}{\sigma_{1}\sigma_{2}}\Kr^{2-m}.
\end{equation*}
\end{proof}

We now compute the $m$-entropy of a $m$-Gaussian and  the $m$-relative entropy between two $m$-Gaussians.

\begin{proposition}
\label{theo:Relative entropy Computation}
We have
\begin{equation}
E_m(\N_m(\boldsymbol{\mu},\Sigma))-E_m(\N_m(\boldsymbol{\mu},V))=
(2-m)C_1(m,d)\Kl\frac{C_0(m,d)}{(\det V)^\frac{1}{2}}\Kr^{1-m}\log_m\frac{(\det V)^\frac{1}{2}}{(\det \Sigma)^\frac{1}{2}},
\end{equation}
and
\begin{equation}
\begin{split}
H_m(\N_m(\boldsymbol{\mu},\Sigma),\N_m(\boldsymbol{\nu},V))=
\frac{1}{2}C_1(m,d)\Kl\frac{C_0(m,d)}{(\det V)^\frac{1}{2}}\Kr^{1-m}
&\left[ \tr\Kl V^{-1}\Sigma\Kr+ \right.\\
&\left. \left\langle\boldsymbol{\mu}-\boldsymbol{\nu},
V^{-1} (\boldsymbol{\mu}-\boldsymbol{\nu})\right\rangle+2\log_{m}\frac{(\det\Sigma)^\frac{1}{2}}{(\det V)^\frac{1}{2}}-d\right]
\end{split}
\end{equation}
\end{proposition}
\begin{remark}
When $m=1$, we recover the corresponding formula for the Gaussian measures.
\end{remark}
\begin{proof}
We denote by $f(\mathbf{x})$ and $g(\mathbf{x})$ respectively the densities of $\N_m(\mathbf{\mu},\Sigma)$ and $\N_m(\mathbf{\nu},V)$. Using the explicit formula of $f$ and $g$ as in~\eqref{eq:qGaussian} and by a straightforward calculation we get
\begin{align}
\int f\log_m f d\mathbf{x}&=\log_m\El\frac{C_0(m,d)}{(\det \Sigma)^\frac{1}{2}}\exp_m\Kl-\frac{d}{2}C_1(m,d)\Kr\Er\nonumber
\\&\overset{\eqref{eq:logq-product}}{=}\Kl-\frac{d}{2}C_1(m,d)\Kr + \log_m\frac{C_0(m,d)}{(\det \Sigma)^\frac{1}{2}} +(1-m)\Kl-\frac{d}{2}C_1(m,d)\Kr\log_m\frac{C_0(m,d)}{(\det \Sigma)^\frac{1}{2}}
\\\int g\log_m g d\mathbf{x}&=\log_m\El\frac{C_0(m,d)}{(\det V)^\frac{1}{2}}\exp_m\Kl-\frac{d}{2}C_1(m,d)\Kr\Er\nonumber
\\&\overset{\eqref{eq:logq-product}}{=}\Kl-\frac{d}{2}C_1(m,d)\Kr + \log_m\frac{C_0(m,d)}{(\det V)^\frac{1}{2}} +(1-m)\Kl-\frac{d}{2}C_1(m,d)\Kr\log_m\frac{C_0(m,d)}{(\det V)^\frac{1}{2}}
\\\int f\log_m g d\mathbf{x}&=\log_m\El\frac{C_0(m,d)}{(\det V)^\frac{1}{2}}\exp_m\Kl-\frac{1}{2}C_1(m,d)(\tr (V^{-1}\Sigma)+\langle\mathbf{\mu}-\mathbf{\nu},V^{-1}(\mathbf{\mu}-\mathbf{\nu}))\rangle\Kr\Er.
\\&\overset{\eqref{eq:logq-product}}{=}\Kl-\frac{1}{2}C_1(m,d)(\tr (V^{-1}\Sigma)+\langle\mathbf{\mu}-\mathbf{\nu},V^{-1}(\mathbf{\mu}-\mathbf{\nu})\rangle\Kr+\log_m\frac{C_0(m,d)}{(\det V)^\frac{1}{2}}
\\&\qquad\qquad +(1-m)\Kl-\frac{1}{2}C_1(m,d)(\tr (V^{-1}\Sigma)+\langle\mathbf{\mu}-\mathbf{\nu},V^{-1}(\mathbf{\mu}-\mathbf{\nu})\rangle)\Kr\log_m\frac{C_0(m,d)}{(\det V)^\frac{1}{2}}
\end{align}
Hence
\begin{align}
\int \Kl f\log_m f-g\log_m g\Kr d\mathbf{x}&=\Kl1-(1-m)\frac{d}{2}C_1(m,d)\Kr\El\log_m\frac{C_0(m,d)}{(\det \Sigma)^\frac{1}{2}}-\log_m\frac{C_0(m,d)}{(\det V)^\frac{1}{2}}\Er\nonumber
\\&=\Kl1-(1-m)\frac{d}{2}C_1(m,d)\Kr\Kl\frac{C_0(m,d)}{(\det V)^\frac{1}{2}}\Kr^{1-m}\log_m\frac{(\det V)^\frac{1}{2}}{(\det \Sigma)^\frac{1}{2}}.
\end{align}
and
\begin{align}
\int (g-f)\log_m gd\mathbf{x}&=\frac{1}{2}C_1(m,d)\Kl \tr (V^{-1}\Sigma)+\langle\mathbf{\mu}-\mathbf{\nu},V^{-1}(\mathbf{\mu}-\mathbf{\nu})\rangle-d\Kr\El1+(1-m)\log_m\frac{C_0(m,d)}{(\det V)^\frac{1}{2}}\Er
\\&=\frac{1}{2}C_1(m,d)\Kl \tr (V^{-1}\Sigma)+\langle\mathbf{\mu}-\mathbf{\nu},V^{-1}(\mathbf{\mu}-\mathbf{\nu})\rangle-d\Kr\El\frac{C_0(m,d)}{(\det V)^\frac{1}{2}}\Er^{1-m}
\end{align}
Since
\begin{equation*}
1-(1-m)\frac{d}{2}C_1(m,d)=(2-m)C_1(m,d),
\end{equation*}
We get
\begin{align}
E_m(\N_m(\boldsymbol{\mu},\Sigma))-E_m(\N_m(\boldsymbol{\mu},V))&=\int \Kl f\log_mf-g\log_mg\Kr d\mathbf{x}\nonumber
\\&=(2-m)\Kl\frac{C_0(m,d)}{(\det V)^\frac{1}{2}}\Kr^{1-m}\log_m\frac{(\det V)^\frac{1}{2}}{(\det \Sigma)^\frac{1}{2}},
\end{align}
and
\begin{align}
H_m(\N_m(\boldsymbol{\mu},\Sigma),\N_m(\boldsymbol{\nu},V))&\overset{\eqref{eq:logq-product}}{=}\frac{1}{2-m}\int \El f\log_mf-g\log_m g-(2-m)\log_m g(f-g)\Er dx\nonumber
\\&=\frac{1}{2}C_1(m,d)\El\frac{C_0(m,d)}{(\det V)^\frac{1}{2}}\Er^{1-m}\El \tr (V^{-1}\Sigma)+\langle\mathbf{\mu}-\mathbf{\nu},V^{-1}(\mathbf{\mu}-\mathbf{\nu})\rangle+2\log_m\frac{(\det V)^\frac{1}{2}}{(\det \Sigma)^\frac{1}{2}}-d\Er.
\end{align}
\end{proof}

\begin{proposition}
\label{theo:optimal relative entropy used for main theorem}
Let $\N_q(\mu_0,C\sigma_0^2), \N_q(\mu,C\sigma^2)$ be given. Set
$\sigma_h^2=(h+\sigma_0^{3-q})^\frac{2}{3-q}$ and let $\N_q(\mu_0,C\sigma_h^2)$ be the solution at time $h$. Let
\begin{equation}
\label{eq:Q_0->h}
Q_{0\Ra h}=\N_{m}(\mu_0,C\sigma_0^2,\mu_0,C\sigma_h^2,\theta_h), \quad \text{where }~ \theta_h=\frac{\sigma_0}{\sigma_h},
\end{equation}
be the $m$-bivariate with mean vector
\begin{equation}
\boldsymbol\mu_{0\Ra h}=\begin{pmatrix} \mu_0\\\mu_0\end{pmatrix},\quad
\Sigma_{0\Ra h}=\begin{pmatrix}
C\sigma_0^2 & C\theta_h\sigma_0\sigma_h\\
C\theta_h\sigma_0\sigma_h & C\sigma_h^2
\end{pmatrix}
=
\begin{pmatrix}
C\sigma_0^2 & C\sigma_0^2\\
C\sigma_0^2& C\sigma_h^2
\end{pmatrix}.
\end{equation}
Define
\begin{equation}
\mathcal{Q}=\left\{\N_{m}(\mu_0,C\sigma_0^2,\mu,C\sigma^2,\theta)\big|\theta\in \RR\right\}
\end{equation}
and
\begin{equation}
Q^*=\mathrm{argmin}_{Q\in\mathcal{Q}}H_{m}(Q\big\|Q_{0\Ra h})
\end{equation}
Then
\begin{equation}
Q^*=\N_{m}(\mu_0,C\sigma_0^2,\mu,C\sigma^2,\eta_h),
\end{equation}
which is a $m$-bivariate with the mean vector and covariance matrix
\begin{equation}
\boldsymbol\mu^*=\begin{pmatrix} \mu_0\\\mu\end{pmatrix},\quad
\Sigma^*=\begin{pmatrix}
C\sigma_0^2 & C\eta_h\sigma_0\sigma\\
C\eta_h\sigma_0\sigma & C\sigma^2
\end{pmatrix},
\end{equation}
where $\eta_h$ satisfies the equation
\begin{equation}
\label{eq:real theta*}
\frac{\eta_h}{(1-\eta_h^2)^\frac{3-m}{2}}=
\frac{\theta_h}{(1-\theta_h^2)^\frac{3-m}{2}}
\Kl\frac{\sigma}{\sigma_h}\Kr^{2-m}.
\end{equation}
Moreover,
the $m$-relative entropy $H_{m}(Q^*\big\|Q_{0\Ra h})$ is
\begin{align}
\label{eq:relative entropy}
&H_{m}(Q^*\big\|Q_{0\Ra h})\nonumber
\\&=\quad\frac{1}{2}C_1(m,2)\Kl\frac{C_0(m,2)}{C\sigma_0(\sigma_h^2-\sigma_0^2)^\frac{1}{2}}\Kr^{1-m}
\El\frac{(\sigma-\sigma_0)^2}{\sigma_h^2-\sigma_0^2}+\frac{2\sigma_0\sigma(1-\eta_h)}{\sigma_h^2-\sigma_0^2}+\frac{(\mu-\mu_0)^2}{C(\sigma_h^2-\sigma_0^2)}+2\log_m\Kl\frac{\sigma_0}{\sigma\eta_h}\Kr^\frac{1}{3-m}-1\Er
\end{align}
The difference $E_q(\N_q(\mu,C\sigma^2))-E_q(\N_q(\mu_0,C\sigma_0^2))$ is
\begin{equation}
\label{eq: difference entropy}
E_q(\N_q(\mu,C\sigma^2))-E_q(\N_q(\mu_0,C\sigma_0^2))=(2-q)C_1(q,1)\Kl \frac{C_0(q,1)}{\sigma_0\sqrt{C}}\Kr^{1-q}\log_q\Kl\frac{\sigma_0}{\sigma}\Kr,
\end{equation}
and the Wasserstein distance $W_2^2(\N_q(\mu,C\sigma^2),\N_q(\mu_0,C\sigma_0^2))$ is
\begin{equation}
\label{eq:Wasserstein distance}
W_2(\N_q(\mu,C\sigma^2),\N_q(\mu_0,C\sigma_0^2))^2=C(\sigma-\sigma_0)^2+(\mu-\mu_0)^2.
\end{equation}
\begin{proof}
By Proposition~\ref{theo:Relative entropy Computation},
\begin{align*}
H_m(Q^*\big\|Q_{0\Ra h})
&=\frac{1}{2}C_1(m,2)\El\frac{C_0(m,2)}{(\det \Sigma_{0\rightarrow h})^\frac{1}{2}}\Er^{1-m}\cdot
\\&\qquad \El \tr (\Sigma_{0\rightarrow h}^{-1}\Sigma^*)+\langle\mathbf{\mu}^*-\mathbf{\mu}_{0\rightarrow h},\Sigma_{0\rightarrow h}^{-1}(\mathbf{\mu}^*-\mathbf{\mu}_{0\rightarrow h})\rangle+2\log_m\frac{(\det \Sigma_{0\rightarrow h})^\frac{1}{2}}{(\det \Sigma^*)^\frac{1}{2}}-2\Er.
\end{align*}

We now calculate each term in the above formula explicitly.

\begin{align*}
&\det \Sigma_{0\rightarrow h}=C^2\sigma_0^2(\sigma_h^2-\sigma_0^2)=C^2\sigma_0^2\sigma_h^2(1-\theta_h^2),
\quad \Sigma_{0\rightarrow h}^{-1}=\frac{1}{C(\sigma_h^2-\sigma_0^2)}
\begin{pmatrix}
\frac{1}{\theta_h^2}& -1\\
-1&1
\end{pmatrix}.
\\&
\det \Sigma^*=C^2\sigma_0^2\sigma^2(1-\eta^2).
\\&
\tr\Kl\Sigma_{0\rightarrow h}^{-1}\Sigma^*\Kr
=\frac{\sigma_h^2+\sigma_0^2-2\eta\sigma_0\sigma}{\sigma_h^2-\sigma_0^2}
=\frac{(\sigma-\sigma_0)^2}{\sigma_h^2-\sigma_0^2}+\frac{2\sigma_0\sigma(1-\eta)}{\sigma_h^2-\sigma_0^2}+1.
\\&
(\mathbf{\mu}^*-\mathbf{\mu}_{0\rightarrow h})^T\Sigma_{0\rightarrow h}^{-1}(\mathbf{\mu}^*-\mathbf{\mu}_{0\rightarrow h})=\frac{(\mu-\mu_0)^2}{C(\sigma_h^2-\sigma_0^2)}.
\\&
\frac{\det\Sigma_{0\rightarrow h}}{\det\Sigma^*}=\frac{\sigma_h^2}{\sigma^2}\cdot\frac{1-\theta_h^2}{1-\eta^2}.
\end{align*}

By~\eqref{eq:real theta*}
\begin{equation*}
\frac{\eta_h}{(1-\eta_h^2)^\frac{3-m}{2}}=
\frac{\theta_h}{(1-\theta_h^2)^\frac{3-m}{2}}
\Kl\frac{\sigma}{\sigma_h}\Kr^{2-m}.
\end{equation*}
Hence
\[
\frac{1-\theta_h^2}{1-\eta_h^2}=\Kl\frac{\theta_h}{\eta_h}\Kr^\frac{2}{3-m}\Kl\frac{\sigma}{\sigma_h}\Kr^\frac{2(2-m)}{3-m},
\]

and
\[
\frac{\sigma_h^2}{\sigma^2}\cdot\frac{1-\theta_h^2}{1-\eta_h^2}=\Kl\frac{\sigma_h\theta_h}{\sigma\eta_h}\Kr^\frac{2}{3-m}=\Kl\frac{\sigma_0}{\sigma\eta_h}\Kr^\frac{2}{3-m}.
\]

Therefore
\[
\frac{(\det\Sigma_{0\rightarrow h})^\frac{1}{2}}{(\det\Sigma^*)^\frac{1}{2}}=\Kl\frac{\sigma_0}{\sigma\eta_h}\Kr^\frac{1}{3-m},
\]

and~\eqref{eq:relative entropy} follows.

\eqref{eq: difference entropy} is a direct consequence of the first equality in Proposition~\ref{theo:Relative entropy Computation} and ~\eqref{eq:Wasserstein distance} has been shown in~\eqref{eq:Wassertein1D}.
\end{proof}
\end{proposition}
\section{Proof of the main theorem}
\label{sec:proof}
In this section, we bring all ingredients together to prove the main theorem. Suppose that the assumption of the main theorem is true.
We set
\begin{equation}
\label{eq:a}
a(q,\sigma_0)=\frac{2C}{C_1(m,2)}\Kl\frac{C_0(m,2)}{C\sigma_0}\Kr^{m-1}=\frac{2C^{2-m}}{C_1(m,2)}\Kl\frac{C_0(m,2)}{\sigma_0}\Kr^{m-1}.
\end{equation}
\begin{equation}
\label{eq:b}
b(\sigma_0,q)=\frac{(2-q)C_1(q,1)}{C}\Kl\frac{C_0(q,1)}{\sigma_0\sqrt{C}}\Kr^{1-q}=\frac{(2-q)C_1(q,1)}{C^\frac{3-q}{2}}\Kl\frac{C_0(q,1)}{\sigma_0}\Kr^{1-q}.
\end{equation}
%

Let $Q^*$ be the minimizer in~\eqref{eq: J_h}. By Proposition~\ref{theo:optimal relative entropy used for main theorem}, we have
\begin{equation*}
Q^*=\N_m(\mu_0,\sigma_0^{2},\mu,\sigma^{2},\eta_h),
\end{equation*}
where $\eta_h=\eta(h,\sigma)$ satisfies the following equation
\begin{equation}
\label{eq:eqn for eta}
\frac{\eta_h}{(1-\eta_h^2)^\frac{3-m}{2}}=
\frac{\theta_h}{(1-\theta_h^2)^\frac{3-m}{2}}
\Kl\frac{\sigma}{\sigma_h}\Kr^{2-m}
=\frac{\frac{\sigma_0}{\sigma_h}}{\Kl1-\frac{\sigma_0^2}{\sigma_h^2}\Kr^\frac{3-m}{2}}\Kl\frac{\sigma}{\sigma_h}\Kr^{2-m}
=\frac{\sigma_0}{(\sigma_h^2-\sigma_0^2)^\frac{3-m}{2}}\sigma^{2-m}.
\end{equation}
Since $\abs{\eta_h}\leq 1$ and the right hand side of~\eqref{eq:eqn for eta} is positive, it holds that $0<\eta_h<1$.
Using the relationship $m=3-\frac{2}{q}$, we can rewrite the above equation as follows
\begin{equation}
\label{eq:eqn for eta 2}
\frac{\eta_h^q}{1-\eta_h^2}=\frac{\sigma_0^q\sigma^{(2-q)}}{\sigma_h^2-\sigma_0^2}.
\end{equation}

We now use the following statement whose proof is straightforward.

Assume that $x_0$ is given. For all $\epsilon>0$, and for all $h>0$ sufficiently small, there exists a constant $C= C(\epsilon,x_0)$ such that
\begin{equation}
\label{eq:statement}
(h+x_0)^\epsilon-x_0^\epsilon\leq Ch.
\end{equation}
In particular, if $\epsilon<1$ then $C=\epsilon x_0^{\epsilon-1}$.

Using~\eqref{eq:statement} for $\epsilon=\frac{2}{3-q}, x_0=\sigma^{3-q}$ and from~\eqref{eq:eqn for eta 2}, we get
\begin{equation}
\label{eq: 1-eta estimate}
1-\eta_h=\frac{\eta_h^q}{1+\eta_h}\frac{\sigma_h^2-\sigma_0^2}{\sigma_0^q\sigma^{(2-q)}}=\frac{\eta_h^q}{1+\eta_h}\frac{(h+\sigma_0^{3-q})^\frac{2}{3-q}-\sigma_0^2}{\sigma_0^q\sigma^{(2-q)}}\leq \frac{Ch}{\sigma^{(2-q)}},
\end{equation}
where $C>0$ is a constant depending only on $\sigma_0$ and $q$. This implies that for fixed $\sigma$, $\lim_{h\rightarrow 0}\eta_h=\lim_{h\rightarrow 0}\eta(h,\sigma)=1$ and as a sequence of functions $\eta(h,\cdot)\rightarrow 1$ locally uniform.



\begin{enumerate}
\item We now prove the first statement of the main theorem. We need to prove
\begin{equation}
a(\sigma_h^2-\sigma_0^2)^\frac{1}{q} J_h(\cdot|\N_0)\overset{\Gamma}{\rightarrow}W_2^2(\cdot,\N_0) \quad\text{as}~ h\rightarrow 0.
\end{equation}
Let $\N_q(\mu,C\sigma^2)$ be given and we denote it by $\N_q$ for short.
By Proposition~\ref{theo:optimal relative entropy used for main theorem}, we have
\begin{equation}
a(\sigma_h^2-\sigma_0^2)^\frac{1}{q}J_h(\N_q|\N_0)=C(\sigma-\sigma_0)^2+(\mu-\mu_0)^2+2C\sigma_0\sigma(1-\eta_h)+C(\sigma_h^2-\sigma_0^2)\Big[2\log_m\Kl\frac{\sigma_0}{\sigma\eta_h}\Kr^\frac{1}{3-m}-1\Big],
\end{equation}
and
\begin{equation}
W_2^2(\N_q,\N_0)=C(\sigma-\sigma_0)^2+(\mu-\mu_0)^2.
\end{equation}
For the lower bound part: Assume that $\N_q^h=\N_q(\nu_h,C\xi_h^2)\rightarrow \N_q$. This means that $(\nu_h-\mu)^2+(\xi_h-\sigma)^2\rightarrow 0$. Hence we can assume that $0<\frac{\sigma}{2}\leq\sup_h\xi_h\leq \frac{3}{2}\sigma$.

Let $\eta_h=\eta(h,\xi_h)$ be the solution of~\eqref{eq:eqn for eta} where $\sigma$ is replaced by $\xi_h$. By~\eqref{eq: 1-eta estimate} we have $\lim_{h\rightarrow}\eta_h=1$, so we can assume that $\xi_h\eta_h$ is uniformly bounded above and away from $0$. We now have

\begin{align*}
&a(\sigma_h^2-\sigma_0^2)^\frac{1}{q}J_h(\N_q^h|\N_0)
\\&\qquad=C(\xi_h-\sigma_0)^2+(\nu_h-\mu_0)^2+2C\sigma_0\xi_h(1-\eta_h)+C(\sigma_h^2-\sigma_0^2)\Big[2\log_m\Kl\frac{\sigma_0}{\xi_h\eta_h}\Kr^\frac{1}{3-m}-1\Big]
\\&\qquad\geq C(\xi_h-\sigma_0)^2+(\nu_h-\mu_0)^2+C(\sigma_h^2-\sigma_0^2)\Big[2\log_m\Kl\frac{\sigma_0}{\xi_h\eta_h}\Kr^\frac{1}{3-m}-1\Big].
\end{align*}
Hence
\begin{align*}
&\liminf_{h\rightarrow 0} a(\sigma_h^2-\sigma_0^2)^\frac{1}{q}J_h(\N_q^h|\N_0)
\\&\qquad\geq \liminf_{h\rightarrow 0} \Big\{C(\xi_h-\sigma_0)^2+(\nu_h-\mu_0)^2+C(\sigma_h^2-\sigma_0^2)\Big[2\log_m\Kl\frac{\sigma_0}{\xi_h\eta_h}\Kr^\frac{1}{3-m}-1\Big]\Big\}
\\&\qquad =C(\sigma-\sigma_0)^2+(\mu-\mu_0)^2=W_2^2(\N_q,\N_0).
\end{align*}
For the upper bound part: as a recovery sequence, we just simply take the fixed sequence $\N_q^h=\N_q$.
\item We now prove the second statement of the main theorem. We need to prove:
\begin{equation}
\label{eq: second statment proof }
ab(\sigma_h^2-\sigma_0^2)^\frac{1-q}{q} J_h(\cdot|\N_0)-\frac{b}{\sigma_h^2-\sigma_0^2}W_2^2(\cdot,\N_0)\overset{\Gamma}{\rightarrow}E_q(\cdot)-E_q(\N_0)\quad\text{as}~ h\rightarrow 0.
\end{equation}
Let $\N_q=\N_q(\mu,\sigma^2)$ be given. Let $\eta_h$ be the solution of~\eqref{eq:eqn for eta}. We have
\begin{align}
ab(\sigma_h^2-\sigma_0^2)^\frac{1-q}{q} J_h(\N_q|\N_0)&=bC\Big\{\frac{(\sigma-\sigma_0)^2}{\sigma_h^2-\sigma_0^2}+\frac{(\mu-\mu_0)^2}{C(\sigma_h^2-\sigma_0^2)}+2\frac{\sigma_0\sigma(1-\eta_h)}{\sigma_h^2-\sigma_0^2}+\Big[2\log_m\Kl\frac{\sigma_0}{\sigma\eta_h}\Kr^\frac{1}{3-m}-1\Big]\Big\},
\\\frac{b}{\sigma_h^2-\sigma_0^2}W_2^2(\N_q,\N_0)&=bC\Big[\frac{(\sigma-\sigma_0)^2}{\sigma_h^2-\sigma_0^2}+\frac{(\mu-\mu_0)^2}{C(\sigma_h^2-\sigma_0^2)}\Big],
\\E_q(\N_q)-E_q(\N_0)&=bC\log_q\Kl\frac{\sigma_0}{\sigma}\Kr.
\end{align}
Define
\begin{align*}
\mathcal{F}_h(\N_q;\N_0)&\colonequals\frac{2\sigma_0\sigma(1-\eta_h)}{\sigma_h^2-\sigma_0^2}+2\log_m\Kl\frac{\sigma_0}{\sigma\eta_h}\Kr^\frac{1}{3-m}-1,
\\\mathcal{F}(\N_q;\N_0)&\colonequals\log_q\Kl\frac{\sigma_0}{\sigma}\Kr.
\end{align*}
We need to prove
\begin{equation}
\mathcal{F}_h(\cdot;\N_0)\xrightarrow{\Gamma}\mathcal{F}(\cdot;\N_0).
\end{equation}
We first prove that $\mathcal{F}_h(\cdot,\N_0)\rightarrow \mathcal{F}(\cdot,\N_0)$ locally uniform. 
We now rewrite the $\mathrm{RHS}$ of $\mathcal{F}_h$ using the relationship between $m$ and $q$.

Since for any $t>0$
\[
\log_mt^\frac{1}{3-m}=\frac{t^\frac{1-m}{3-m}-1}{1-m}=\frac{t^{1-q}-1}{\frac{2(1-q)}{q}}=\frac{q}{2}\log_qt;
\]
Hence
\begin{equation}
\label{eq:first term of F_h}
2\log_m\Kl\frac{\sigma_0}{\sigma\eta_h}\Kr^\frac{1}{3-m}=q\log_q\Kl\frac{\sigma_0}{\sigma\eta_h}\Kr.
\end{equation}
From~\eqref{eq:eqn for eta}, we get
\begin{equation}
\label{eq:prelimit}
\frac{\eta_h}{(1+\eta_h)^\frac{3-m}{2}}=\Kl\frac{\sigma_0\sigma(1-\eta_h)}{\sigma_h^2-\sigma_0^2}\Kr^\frac{3-m}{2}\Kl\frac{\sigma}{\sigma_0}\Kr^\frac{1-m}{2}.
\end{equation}
This implies that
\begin{align}
\label{eq:second term of F_h}
\frac{\sigma_0\sigma(1-\eta_h)}{\sigma_h^2-\sigma_0^2}=\frac{\eta_h^q}{1+\eta_h} \Kl\frac{\sigma_0}{\sigma}\Kr^{1-q}=\frac{\eta_h^q}{1+\eta_h}\Kl(1-q)\log_q\Kl\frac{\sigma_0}{\sigma}\Kr+1\Kr.
\end{align}

From~\eqref{eq:first term of F_h} and~\eqref{eq:second term of F_h} we obtain
\begin{equation*}
\mathcal{F}_h(\N_q;\N_0)=2\frac{\eta_h^q}{1+\eta_h} \Kl\frac{\sigma_0}{\sigma}\Kr^{1-q}+q\log_q\Kl\frac{\sigma_0}{\sigma\eta_h}\Kr-1.
\end{equation*}
Now we have the following estimate
\begin{align}
\abs{\mathcal{F}_h-\mathcal{F}}&=\Big|2\frac{\eta_h^q}{1+\eta_h} \Kl\frac{\sigma_0}{\sigma}\Kr^{1-q}+q\log_q\Kl\frac{\sigma_0}{\sigma\eta_h}\Kr-\log_q\Kl\frac{\sigma_0}{\sigma}\Kr-1\Big|\nonumber
\\&=\Big|2\frac{\eta_h^q}{1+\eta_h} \Kl\frac{\sigma_0}{\sigma}\Kr^{1-q}-1-(1-q)\log_q\Kl\frac{\sigma_0}{\sigma}\Kr+q\log_q\Kl\frac{\sigma_0}{\sigma\eta_h}\Kr-q\log_q\Kl\frac{\sigma_0}{\sigma}\Kr\Big|\nonumber
\\&=\Kl\frac{\sigma_0}{\sigma}\Kr^{1-q}\Big| 2\frac{\eta_h^q}{1+\eta_h}-1+\frac{q}{1-q}\Kl\eta_h^{(q-1)}-1\Kr\Big|\nonumber
\\&\leq\Kl\frac{\sigma_0}{\sigma}\Kr^{1-q}\Big[\Big|2\frac{\eta^q}{1+\eta_h}-1\Big|+\Big|\frac{q}{1-q}\Kl\eta_h^{(q-1)}-1\Kr\Big|\Big]\nonumber
\\&\leq \frac{1}{\sigma^{1-q}}C(1-\eta_h)\nonumber
\\&\overset{\eqref{eq: 1-eta estimate}}\leq \frac{Ch}{\sigma^{3-2q}},\label{eq:estimate}
\end{align}
where $C$ is a constant depending only on $\sigma_0$ and q.


The locally uniform convergence of $\mathcal{F}_h$ to $\mathcal{F}$ thus follows from the estimate~\eqref{eq:estimate}.

For the $\Gamma$-convergence, we get the lower bound part by the local uniform convergence and the continuity of the entropy. Indeed, let $\N_q\colonequals\N_q(\mu,C\sigma^2)$ be given and assume that $\N_q^h\colonequals\N_q(\mu_h,C\xi_h^2)\rightarrow\N_q$. Then $\mu_h\rightarrow\mu$ and $\xi_h\rightarrow\sigma$. Hence we can assume without loss of generality that $\frac{\sigma}{2}\leq\sup_h\xi_h\leq\frac{3\sigma}{2}$. We have the following estimate
\begin{align*}
\Big|\mathcal{F}_h(\N_q^h;\N_q^0)-\mathcal{F}(\N_q;\N_q^0)\Big|
&\leq
\Big|\mathcal{F}_h(\N_q^h;\N_0)-\mathcal{F}(\N_q^h;\N_q^0)\Big|+
\Big|\mathcal{F}(\N_q^h;\N_0)-\mathcal{F}(\N_q;\N_q^0)\Big|
\\&\overset{\eqref{eq:estimate}}\leq \frac{Ch}{\xi_h^{3-2q}}+\Big|\log_q\Kl\frac{\sigma_0}{\xi_h}\Kr-\log_q\Kl\frac{\sigma_0}{\sigma}\Kr\Big|
\\&\leq \frac{Ch}{\sigma^{3-2q}}+\Big|\log_q\Kl\frac{\sigma_0}{\xi_h}\Kr-\log_q\Kl\frac{\sigma_0}{\sigma}\Kr\Big|\rightarrow 0.
\end{align*}
Therefore
\begin{equation*}
\lim_{h\rightarrow 0} \mathcal{F}_h(\N_q^h;\N_q^0)=\mathcal{F}(\N_q;\N_q^0).
\end{equation*}
For the upper part, as a recovery sequence, we can choose the fixed sequence $\N_q^h=\N_q$.
\item We now prove the third statement of the main theorem. Assume that $0<q< 1$. We need to prove

\[
ab(\sigma_h^2-\sigma_0^2)^\frac{1-q}{q} J_h(\cdot|\N_0)-\frac{1}{2h}W_2^2(\cdot,\N_0)\xrightarrow{\Gamma-\liminf} E_q(\cdot)-E_q(\N_0)\quad\text{as}~ h\rightarrow 0.
\]
We have
\begin{align*}
&ab(\sigma_h^2-\sigma_0^2)^\frac{1-q}{q} J_h(\cdot|\N_0)-\frac{1}{2h}W_2^2(\cdot,\N_0)
\\\qquad &=ab(\sigma_h^2-\sigma_0^2)^\frac{1-q}{q} J_h(\cdot|\N_0)-\frac{b}{\sigma_h^2-\sigma_0^2}W_2^2(\cdot,\N_0)+\Kl\frac{b}{\sigma_h^2-\sigma_0^2}-\frac{1}{2h}\Kr W_2^2(\cdot,\N_0).
\end{align*}
Since $0<q<1, 0<\frac{2}{3-q}<1$, using~\eqref{eq:estimate} for $\epsilon=\frac{2}{3-q},~x_0=\sigma_0^{3-q}$, we have
\[
\sigma_h^2-\sigma_0^2=(h+\sigma_0^{3-q})^\frac{2}{3-q}-\sigma_0^2\leq \frac{2}{(3-q)\sigma_0^{1-q}}h.
\]
Therefore
\[
\frac{b}{\sigma_h^2-\sigma_0^2}\geq \frac{(3-q)b\sigma_0^{1-q}}{2h}=\frac{1}{2h}.
\]
It implies that
\[
ab(\sigma_h^2-\sigma_0^2)^\frac{1-q}{q} J_h(\cdot|\N_0)-\frac{1}{2h}W_2^2(\cdot,\N_0)
\geq ab(\sigma_h^2-\sigma_0^2)^\frac{1-q}{q} J_h(\cdot|\N_0)-\frac{b}{\sigma_h^2-\sigma_0^2}W_2^2(\cdot,\N_0),
\]
and the third statement thus follows from the second one.
\end{enumerate}
This completes the proof of the main theorem.
\begin{center}
\textbf{Acknowledgement}
\end{center}
The research of the paper has received funding from the ITN ``FIRST" of the Seventh Framework Programme of the European Community (grant agreement number 238702). This project was done partially when the author was in the University of Bath as part of the ITN program. The author would like to thank Vaios Laschos for his helpful discussion.
\bibliographystyle{alpha}	
\bibliography{PMEbib}

\end{document}